\documentclass[12pt, oneside, reqno]{amsart}
\usepackage{amsmath,amssymb,amsthm,graphicx,mathrsfs,url}
\usepackage{enumitem}
\usepackage[usenames,dvipsnames]{color}
\usepackage[colorlinks=true,linkcolor=Red,citecolor=PineGreen]{hyperref}

\newcommand{\C}{\mathbb{C}}
\newcommand{\Z}{\mathbb{Z}}
\newcommand{\R}{\mathbb{R}}
\DeclareMathOperator{\WF}{\mathrm{WF}}
\DeclareMathOperator{\tr}{\mathrm{tr}}
\DeclareMathOperator{\Op}{\mathrm{Op}}
\let\Re\undefined
\DeclareMathOperator{\Re}{\mathrm{Re}}

\newcommand{\dd}{\mathrm{d}}

\newtheorem{theorem}{Theorem}
\newtheorem{proposition}[theorem]{Proposition}
\newtheorem{corollary}[theorem]{Corollary}
\newtheorem{lemma}[theorem]{Lemma}
\theoremstyle{definition}

\begin{document}

\begin{abstract}
We study the wavefront set of resolvent of arbitrary flows in their region of convergence, to obtain a general formula for their intersection with currents. We provide an application to the topology of surfaces. 
\end{abstract}

\title{Resolvent of vector fields and Lefschetz numbers}
\author{Yann Chaubet and Yannick Guedes Bonthonneau}
\maketitle

Given a chaotic flow $(\varphi_t)$ acting on a closed manifold $M$, the distribution of periodic points, given as the solutions $(t, z) \in \R_{>0} \times M$ of the equation
\begin{equation}\label{eq:periodicequation}
\varphi_t(z) = z
\end{equation}
have been extensively studied in the literature. To study the periodic orbits, Smale \cite{Smale-67} suggested to introduce dynamical series or zeta functions involving their travel time; those series have proved to be very effective and have given rise to many important works, starting from the seminal paper of Ruelle \cite{ruelle1976zeta}. The last decade has seen a series of breakthrough in this topic, in particular with \cite{giulietti2013anosov}, \cite{dyatlov2016dynamical}, \cite{Faure-Tsujii-17} and \cite{DyGu18}. 

The modern approach for studying the analytic properties of a dynamical series consists in rewriting it as a distributional pairing or trace involving the resolvent
\[
(s - \mathcal L_X)^{-1} = \int_0^\infty e^{-ts} \varphi_{-t}^* \dd t \ : \ \Omega^\bullet(M) \to \mathcal{D}'^\bullet(M)
\]
of the flow. Here $\Omega^\bullet(M)$ is the space of smooth differential forms on $M$ while $\mathcal{D}'^\bullet(M)$ is its dual topological space, $s$ is a complex parameter with large real part, $\varphi_t^*$ is the pull-back by the flow and $\mathcal L_X$ is the Lie derivative in the direction of the generator $X = \left.\frac{\dd}{\dd t} \varphi_t\right|_{t = 0}$ of $(\varphi_t)$. However, establishing such a relation requires some justification, and authors often deal with this crucial step in an ad-hoc fashion (see \S\ref{sec:results} below). For this reason, we propose in this paper to gather and generalize these statements by providing a formula which expresses dynamical series of a general form in terms of the resolvent of any flow (not necessarily chaotic), see Theorem \ref{thm:currents} below. Before stating our results in their most precise form in \S\ref{sec:results}, we first present an application about topology of surfaces.

It is somewhat natural to ask what happens if we replace \eqref{eq:periodicequation} by the equation
\begin{equation}\label{eq:Fperiodic}
\varphi_t(z) = F(z)
\end{equation}
where $F : M \to M$ is a smooth map. While this question can be asked in a very general context, we will focuse on geodesic flows of surfaces. To be more precise, let $\Sigma$ be a closed oriented Riemannian surface whose geodesic flow $(\varphi_t)_{t \in \R}$, acting on the unit tangent bundle $M = S\Sigma$, has the Anosov property. The geodesic vector field $X$ is the Reeb vector field of the Liouville contact one-form $\alpha \in \Omega^1(M)$. We denote by $\pi :S\Sigma\to \Sigma$ the natural projection and by $(\Phi_t)_{t\in\R}$ the Hamiltonian bundle lift of $(\varphi_t)_{t \in \R}$ to $T^\ast S\Sigma$.

Next, let $f:\Sigma\to\Sigma$ be a diffeomorphism and $F:S\Sigma\to S\Sigma$ be a bundle map over $f$, in the sense that $\pi \circ F = f \circ \pi$. We will say that $F$ is \textit{transverse} to $(\varphi_t)$ if
\begin{equation}\label{eq:transverse}
N^*\mathrm G_F \cap \Gamma(\varphi)' = \{0\}.
\end{equation}
Here $N^* \mathrm G_F \subset T^*(M \times M)$ denotes the conormal bundle of the graph $\mathrm G_F = \{(F(z), z))~|~z \in M\}$ of $F$, while $\Gamma(\varphi) = \Delta(T^*M) \cup \overline{\Omega_+}$, where $\Delta(T^*M)$ is the diagonal in $T^*(M \times M)$ and 
\[
\Omega_+ = \bigl\{(\Phi_t(z,\xi),(z,\xi))\ |\ (z,\xi) \in T^*M,~\langle \xi, X(z)\rangle = 0,~t \geqslant 0\bigr\}.
\]
Finally, we denoted $A' = \{ ((y,\eta),(x,-\xi))\ |\ ((y,\eta),(x,\xi))\in A\}$ for each $A\subset T^\ast(M\times M)$. 

Under the condition \eqref{eq:transverse}, the set
\[
Z_F = \bigl\{(t, z) \in \R_+ \times M \ |\ \varphi_t(z) = F(z)\bigr\},
\]
which is the set of points satisfying \eqref{eq:Fperiodic}, is an oriented $1$-dimensional submanifold of $\R_+ \times M$.

\begin{theorem}\label{thm:lefschetz}
Let $f$ and $F$ be as above. For $\Re s  \gg 1$ the integral
\begin{equation}\label{eq:eta(s)}
\eta(s) = \int_{Z_F} e^{- t s}  \pi_M^* \alpha
\end{equation}
converges, where $\pi_M : \R_+ \times M \to M$ is the projection over the second factor. Moreover $s \mapsto \eta(s)$ has a meromorphic continuation to $\C$ and
\begin{equation}\label{eq:eta(s)residue}
\lim_{s\to 0} s\eta(s) = \Lambda_f + \int_{S\Sigma} (\iota_X F_\ast \alpha   - 1) \dd \mu,
\end{equation}
where $\Lambda_f$ is the Lefschetz number of $f$. Here $\mu$ is the normalized measure on $M$ induced by $\alpha \wedge \dd \alpha$, and $\iota_X$ is the interior product with $X$.
\end{theorem}

The study of the values of dynamical zeta functions at zero and their potential relation with topological invariants started with Fried \cite{Fried-87} and has recently seen much activity, with for example \cite{Dyatlov-Zworski-17}, \cite{Dang-Riviere}, \cite{DGRS-17}, \cite{Chaubet-Dang-19}, \cite{Paternain-Cekic} or \cite{Shen-21} to cite just a few.

We observe that for any small perturbation $f$ of the identity map, one can find a bundle map $F$ transverse to the flow. Moreover, we will show in \S\ref{sec:examples} that one can build examples for which $f$ is not isotopic to the identity map. 

Note also that Theorem \ref{thm:lefschetz} extends the main result of \cite{Dyatlov-Zworski-17}, in which the authors compute the order of vanishing at zero of the Ruelle zeta function associated to a negatively curved surface. Indeed, it is not hard to see that Theorem~\ref{thm:lefschetz} also holds if we replace $F$ by $\varphi_T \circ F$ for any $T \in \R$. The series \eqref{eq:eta(s)} associated to the maps $f = \mathrm{id}$ and $F = \varphi_{-\varepsilon}$ (note that $\varphi_{-\varepsilon}$ is transverse to the flow whenever $\varepsilon > 0$ is small enough) reads
$$
\eta(s) = e^{\varepsilon s} \sum_\gamma \ell^\sharp(\gamma) e^{-s \ell(\gamma)}
$$
where the sum runs over every closed geodesics $\gamma$ of $\Sigma$, and $\ell(\gamma)$, $\ell^\sharp(\gamma)$ denote the length and the primitive length of $\gamma$, respectively. Then \eqref{eq:eta(s)residue} yields $s \eta(s) \to 2 - 2\mathrm g$ as $s\to 0$ which is \cite[Theorem p.1]{Dyatlov-Zworski-17}. 

To further illustrate Theorem \ref{thm:lefschetz}, we provide another simple example for which $F$ is explicit. Let $f = \mathrm{id}$ and $F = R_\theta$ for $\theta \in \R \setminus \pi \Z$, where $R_\theta : S\Sigma \to S\Sigma$ is the rotation of angle $\theta$ in the fibers of $S\Sigma$. Then $R_\theta$ is transverse to the geodesic flow and we denote by $\eta_\theta(s)$ the dynamical series \eqref{eq:eta(s)} associated with $F = R_\theta$. This series involve the variety of
vectors $(x,v) \in S\Sigma$ such that for some $t$ it holds 
$$
\pi(\varphi_t(x,v)) = x \quad \text{and} \quad \angle~(v, \varphi_t(x,v)) = \theta.
$$
Then using Theorem \ref{thm:lefschetz} we find that
\begin{equation}\label{eq:residuetheta}
\lim_{s \to 0} s \eta_\theta(s) = 1 + \cos \theta - 2 \mathrm g,
\end{equation}
where $\mathrm g$ is the genus of $\Sigma$. If moreover $\Sigma$ is hyperbolic then the series $\eta_\theta(s)$ can be computed explicitly and reads
\[
\eta_\theta(s) = \cos(\theta / 2) \sum_{\gamma} \ell^\sharp(\gamma) e^{-s \varrho(\theta,\,\ell(\gamma))} \sqrt{1 + \frac{\tan^2 \theta / 2}{\tanh^2 \ell(\gamma) / 2}},
\]
where the sum runs over all closed geodesics $\gamma$ of $\Sigma$, $\ell(\gamma)$ (resp. $\ell^\sharp(\gamma)$) is the length (resp. primitive length) of $\gamma$ and
\[
\varrho(\theta, \ell) = 2 \cosh^{-1}\left(\frac{\cosh \ell / 2}{\cos \theta / 2}\right).
\]
In fact, in that case one can actually show that
\[
\eta_\theta(s) = \eta_0(s) + 2 \sin^2(\theta / 2)\, \eta_0(s + 1) + g(s)
\]
where $g$ is analytic in  $\{\Re s > - 1\};$ then \eqref{eq:residuetheta} can be deduced from the fact that $s\eta_0(s) \to 2 - 2 \mathrm g$ and $s \eta_0(s + 1) \to -1$ as $s \to 0$, as it follows from the results of Selberg. Note that it is also plausible that \eqref{eq:residuetheta} could be deduced from Selberg's trace formula. However, those alternative approaches would involve additional computations and they would not lead to a general result involving $\eta_\theta(s)$ for surfaces with \textit{variable} negative curvature.

As mentioned above, Theorem \ref{thm:lefschetz} is a consequence of more technical results which allow to interpret the series \eqref{eq:eta(s)} as a distributional pairing involving the Schwartz kernel of the resolvent of the geodesic flow. We detail those results in the next paragraph. 

\numberwithin{equation}{section}
\section{The wave front set of the resolvent}
\label{sec:results}

In this section we assume that $X$ a smooth vector field on a closed and oriented manifold $M$ of dimension $n$. Denote by $\Omega^k(M)$ the space of smooth differential $k$-forms and by $\mathcal{D}'^{n-k}(M)$ the space of $(n-k)$-currents, that is, the topological dual space of $\Omega^k(M)$. Then we have $\Omega^\bullet(M) = \oplus_k \Omega^k(M)$ and $\mathcal{D}'^\bullet(M) = \oplus_k \mathcal{D}'^k(M)$. For any operator
\[
A : \Omega^\bullet(M) \to \mathcal{D}'^\bullet(M)
\]
preserving the degree, its Schwartz kernel $\mathrm{K}_A \in \mathcal{D}'^n(M \times M)$ is defined by
\[
\int_M (A v) \wedge w = \int_{M \times M}\mathrm{K}_A \wedge \pi_1^*v \wedge \pi_2^*w, \quad v,w \in \Omega^\bullet(M),
\]
where $\pi_j : M \times M \to M$ is the projection over the first and second factors, respectively.
We can now state our main result
\begin{theorem}\label{thm:main}
Let $J(t)$ be the Schwartz kernel of $e^{tX} = \varphi_t^*$ and $\mathrm{K}(s)$ that of $(s-\mathcal L_X)^{-1}$. Let $u\in \mathcal{D}'^k(M\times M)$ such that 
\[
\WF'(u)\cap \Gamma(\varphi) = \emptyset.
\]
Then the product $\mathrm{K}(s)\wedge u$ is well defined and for any $\chi \in C^\infty_c(\R)$ which is equal to $1$ near $0$, we have, for $\Re s$ large,
\[
\lim_{\varepsilon \to 0 }\left(\int_{\R^+} \chi(t/\varepsilon) e^{-t s} J(t) \dd t \right)\wedge u =\mathrm{K}(s)\wedge u,
\]
where the convergence holds in $\mathcal{D}'^{n + k}(M \times M).$ Moreover, the convergence is uniform as long as the wavefront set of $u$ lies in a fixed compact subset not intersecting $\Gamma(\varphi)$, with uniform estimates.
\end{theorem}

Before turning to the proof, let us discuss the main application of this statement. We define the left flow $(\varphi^\mathrm L_t)$ acting on $M \times M$ by
\[
\varphi^\mathrm L_t(x,y) = (\varphi_t(x), x), \quad x, y \in M,
\]
denote by $X^{\mathrm L}$ its generator and set
\[
G_\varphi = \{(t, \varphi_t(x), x)\ |\ t \geqslant 0,\ x \in M\} \subset \R \times M \times M.
\]
Finally, if $P \subset M_1$ is an oriented submanifold of a manifold $M_1$, we will denote by $[P] \in \mathcal{D}'^\bullet(M_1)$ the associated integration current, which is defined by 
$$
\langle [P], \omega \rangle = \int_{P} \iota_P^*\, \omega, \quad \omega \in \Omega^\bullet(M_1)
$$
where $\iota_P : P \hookrightarrow M_1$ is the inclusion.

The following result allows to express general dynamical series in terms of the resolvent.

\begin{theorem}\label{thm:currents}
Assume that ${\mathrm{Q}} \subset M\times M$ is a submanifold of dimension $d \geqslant n$, and that its conormal $N^\ast({\mathrm{Q}})$ intersects $\Gamma(\varphi)'$ trivially. Then the following holds. \vspace{3pt}
\begin{enumerate}[label=$\mathrm{(\roman*)}$]
\item \label{item:i} Let $\tau : \R \times M^2 \to \R$ and $\pi : \R \times M^2 \to M^2$ be the projections over the first and second factor, respectively. Then, $\pi^{-1}({\mathrm{Q}})$ intersects $G_\varphi$ transversally, so that $G_\varphi \cap \pi^{-1}({\mathrm{Q}})$ is a smooth submanifold of $\R_+ \times M^2$.
\item \label{item:ii} For large $\Re s$, the wedge product $ \iota_{X^\mathrm L}\mathrm{K}(s) \wedge [{\mathrm{Q}}]$ is well defined, the current $e^{-s\tau}\left[G_\varphi \cap \pi^{-1}({\mathrm{Q}})\right]$ lies in the dual of $\pi^*\Omega^\bullet (M)$, and we have
\[
\iota_{X^\mathrm L}\mathrm{K}(s) \wedge [{\mathrm{Q}}] = \pi_*\left(e^{-s\tau} \left[G_\varphi \cap \pi^{-1}({\mathrm{Q}})\right] \right),
\]
where the equality holds in $\mathcal{D}'^{3n - d - 1}(M\times M).$ 
\end{enumerate}\vspace{6pt}
\end{theorem}

Let us briefly explain how Theorem \ref{thm:currents} can be used to recover some known formulae in two important situations.

First, for $\varepsilon > 0$, let ${\mathrm{Q}} = \Delta_{-\varepsilon} = \{(\varphi_{-\varepsilon}(x), x)~:~x \in M\}$. Assume that $N^*{\mathrm{Q}} \cap \Gamma(\varphi) = \{0\}$, which is always satisfied for example if $(\varphi_t)$ is Anosov and $\varepsilon$ is small enough. Then \ref{item:ii} yields that
\[
 \iota_{X^\mathrm L}\mathrm{K}(s) \wedge [\Delta_{-\varepsilon}] = e^{\varepsilon s} \sum_{\gamma} e^{-s \tau(\gamma)} [\Lambda_\gamma^\varepsilon],
\]
where the sum runs over all periodic orbits of the flow, and 
\[
\Lambda_\gamma^\varepsilon = \{(\varphi_{t-\varepsilon}(x), \varphi_t(x))\ |\ t \in [0, \tau^\sharp(\gamma)]\} \subset M \times M.
\]
Here $\tau^\sharp(\gamma)$ is the primitive period of $\gamma$ and $x$ is any point in the image of $\gamma$. In particular, if $\omega \in \Omega^1(M \times M)$ is any smooth one form such that $\iota_{X \times X} \omega = 1$, then we get
\begin{equation}\label{eq:seriesperiodic}
\langle\mathrm{K}(s) \wedge \iota_{X^\mathrm{R}} [\Delta_{-\varepsilon}], \omega \rangle = e^{\varepsilon s} \sum_\gamma \tau^\sharp(\gamma) e^{-s \tau(\gamma)}.
\end{equation}
The above series is, up to the term $e^{\varepsilon s}$, the logarithmic derivative of the so-called Ruelle zeta function, which was shown to admit a meromorphic continuation to $\C$ in \cite{giulietti2013anosov}, and later in \cite{dyatlov2016dynamical} with semi-classical methods. We emphasize that the convergence of the sum in the right-hand side of \eqref{eq:seriesperiodic} is a consequence of our results and yields an upper bound on the growth of periodic orbits. On the contrary, in \cite{dyatlov2016dynamical}, Equation \eqref{eq:seriesperiodic}, which appears as a crucial step in the proof of the meromorphic continuation, is proved by using \textit{a priori} estimates on the growth of periodic orbits\footnote{In \cite[\S4]{dyatlov2016dynamical}, the authors express the logarithmic derivative of the Ruelle zeta function as the  as the super flat trace of the operator $\varphi_{\varepsilon}^*(s-\mathcal L_X)^{-1}$ acting on $\ker \iota_X$; however this super flat trace coincides with the left-hand side of \eqref{eq:seriesperiodic} by Lemma \ref{lem:simplecomputation} below.}.

Next, assume that ${\mathrm{Q}} = {\mathrm{Q}}_1 \times {\mathrm{Q}}_2$ where $\dim {\mathrm{Q}}_1 + \dim {\mathrm{Q}}_2 + 1 = n$, where ${\mathrm{Q}}_{j}$ is transverse to the flow for $j = 1,2$. Then \ref{item:ii} gives that
\[
\mathrm{K}(s) \wedge \iota_{X^\mathrm{R}} [{\mathrm{Q}}] = \sum_{\substack{t \geqslant 0 \\ (x_1, x_2) \in {\mathrm{Q}}_1 \times {\mathrm{Q}}_2 \\ \varphi_t(x_1) = x_2}} e^{-st} [(x_1, x_2)].
\]
Again, we get
\begin{equation}\label{eq:seriesdang}
\langle \iota_{X^\mathrm L}\mathrm{K}(s) \wedge [{\mathrm{Q}}], 1 \rangle = \sum_\gamma e^{-s \tau(\gamma)},
\end{equation}
where the sum runs over all dynamical arcs $\gamma$ joining ${\mathrm{Q}}_1$ and ${\mathrm{Q}}_2$. For Anosov flows, the series in the right-hand side was studied in \cite{Dang-Riviere} and Equation \eqref{eq:seriesdang} is essentially \cite[Theorem 4.15]{Dang-Riviere}. A similar statement for geodesic flows of surfaces with boundary was obtained later by the first author \cite[Proposition~3.4]{chaubet2022poincare} using methods in the same spirit of those we will present below. In the aforementioned works, it is shown that the value at zero of certain Poincar\'e series counting geodesic arcs linking two points in a surface coincides with the inverse of the Euler characteristic of the surface and formula \eqref{eq:seriesdang} is the starting point of the proof.

Another consequence of Theorem \ref{thm:currents} is of course Theorem \ref{thm:lefschetz}, whose proof heavily relies on the method employed in \cite{Dyatlov-Zworski-17}.
\\

The paper is organized as follows. In the next section, we will introduce our main technical tool, and give its proof. We will assume that the reader is familiar with the notions of wavefront set, and proficient in the use of (semi-classical) pseudo-differential operators. We will prove Theorems \ref{thm:main} and \ref{thm:currents} in paragraphs \S\ref{sec:proof-main} and \S\ref{sec:proof-integration-current}, respectively. Finally, in sections \S\ref{sec:proof-lefschetz} and \S\ref{sec:examples}, we prove Theorem \ref{thm:lefschetz} as well as the existence of several non trivial examples satisfying our transversality assumptions. 

\section{Egorov for flows (semi-classical version)}

That the conjugation of a pseudo-differential operator by a (smooth) flow still is pseudo-differential follows from the invariance of the class of pseudo-differential operators under the action of diffeomorphisms. The most usual proof of this fact relies on the so-called Kuranishi trick. While this remains true when the time is large, it is only a qualitative statement. In this section, we will give a quantitative result. 

For $X$ a smooth vector field on a compact manifold $M$ with flow $\varphi_t$, we denote by $\lambda_{\mathrm{max}}(X)$ is maximal Lyapunov exponent, i.e
\[
\lambda_{\mathrm{max}}(X) = \limsup_{t\to \infty}\frac{1}{t}\log \sup_z \max( \|\dd\varphi_t(z)\|, \|\dd\varphi_{-t}(z)\| ). 
\]
We will need to quantify the action on $L^2$, Setting
\[
\mu_{\mathrm{max}}(X)= \limsup_{t\to\pm \infty} \frac{1}{t} \log \|e^{tX}\|_{L^2}. 
\]
Also we define $\Phi = (\Phi_t)_{t \in \R}$ the Hamiltonian lift of $(\varphi_t)$, acting on $T^*M$, which is defined by
\[
\Phi_t(z, \xi) = \dd \varphi_t(z)^{-\top} \xi, \quad (z, \xi) \in T^*M,
\]
where $^{-\top}$ denotes the inverse transpose. Let us introduce some usual classes of pseudo-differential operators. The ones whose symbols satisfy $S^m$ estimates in local charts
\[
|\partial_z^\alpha \partial_\xi^\beta a| \leqslant C_{\alpha,\beta} \langle \xi\rangle^{m-|\beta|},
\]
form the algebra $\Psi^0$. When working with $h$-pseudo-differential operators, for $0\leqslant \delta < 1/2$, we will write $A\in \Psi^m_{\delta,\, \mathrm{sc}}$ if its symbol satisfies the $S^m_\delta$ estimates
\[
|\partial_z^\alpha\partial_\xi^\beta a| \leqslant C_{\alpha,\beta} h^{-\delta(|\alpha|+|\beta|)} \langle\xi\rangle^{m-|\beta|}. 
\]
We can now state our result. 
\begin{theorem}\label{thm:egorov}
Let $X$ be a smooth vector field, and $A\in \Psi^0$ with principal symbol $a\in S^0$. Let $\lambda>\lambda_{\mathrm{max}}(X)$ and $\mu>\mu_{\mathrm{max}}(X)$. Given $0\leqslant \delta < 1/2$, for $h>0$ and $|t| \leqslant \delta |\log h|/\lambda$, we can find $A_t$ and $R_t$ such that
\[
e^{tX} A e^{-tX} = A_t + R_t, 
\]
where $R_t$ is smoothing, with estimates $\| R_t \|_{H^{-N}\to H^N} = \mathcal{O}(h^{-3N-\mu/\lambda})$, and $A_t\in \Psi^0_{\delta,\mathrm{sc}}$ is a $h$-pseudo-differential operator with principal symbol satisfying $a_t = a\circ \Phi_t \mod \langle\xi\rangle^{-1} $, and $\WF(A_t) = \Phi_{-t}( \WF(A))$.
\end{theorem}

\begin{proof}
We start by choosing a Weyl quantization procedure $\Op$ --- as in \S14 of \cite{Zworski-book} --- and write $A = \Op(a) + R_0$, $R_0$ being a $C^\infty$ smoothing operator. Next, we introduce a positive parameter $h>0$ and a cutoff, writing $a= a \chi + a (1-\chi)$, so that $a \chi$ is supported in $|\xi|\leqslant 1/h$, and $a(1-\chi)$ in $|\xi|> 1/2h$. 

By rescaling, we observe that $\Op(a(1-\chi))\in \Psi^0_{0,\, \mathrm{sc}}$ of the form $\Op_h(b)$ with 
\[
b(z,\xi) = [a(1-\chi)](z, \xi/h). 
\]

Using the strategy of \cite{Bouzouina-Robert-02}, we then obtain that as long as $0\leqslant \delta <1/2$ and $|t|\leqslant \delta|\log h|/ \lambda$, we can find a symbol $b_t$ satisfying $S_\delta^0$ estimates and such that 
\[
e^{tX} \Op_h(b) e^{-tX} = \Op_h(b_t) + \mathcal{O}_{\mathrm{smoothing}}(h^\infty). 
\]
We will not detail the arguments of \cite{Bouzouina-Robert-02}, which are presented originally for operators in $\R^n$, with slightly different classes of symbols. For an exposition on surfaces, we refer to the appendix A of \cite{Dyatlov-Jin-Nonnenmacher-19}. The two main ingredients in the proof are the following
\begin{itemize}
	\item For the range of times allowed, $b\circ \Phi_t\in S_\delta^0$ uniformly. 
	\item We have a product formula, for $c\in S^m_\delta$:
	\[
	[X, \Op_h(c)] = \frac{h}{i} \Op_h( \{ p, c \}) + \mathcal{O}(h^{2-2\delta} \Psi^{m-2}_{\delta,\mathrm{sc}}). 
	\]
\end{itemize}
It is a feature of the proof that
\[
b_t = b\circ \Phi_t + \mathcal{O}(h^{1-\delta} S^{-1}_\delta),
\]
and $b_t$ shares the support of $b\circ \Phi_t$. 

Now, $\Op(a\chi)$ is \emph{not} semi-classical, so that there is not much microlocal information we can find on $e^{tX} \Op(a \chi) e^{-tX}$ for $t$ large. The best we can do is thus estimating $H^s\to H^r$ norms. Letting 
\[
R_0' = R_0+ \Op(a\chi),
\]
we get $\| R_0' \|_{H^s\to H^r} = \mathcal{O}( h^{s-r} )$ as $h\to 0$ when $r\geqslant s$. Then in the range $|t|\leqslant \delta|\log h|/\lambda$, we find
\begin{align*}
\| e^{tX} R_0' e^{-tX} \|_{H^{-N}\to H^N} &\leqslant C_N h^{-2N} \| e^{-tX} \|_{H^N} \| e^{tX} \|_{H^{-N}} \\
										&\leqslant C h^{-2N} e^{2N\lambda |t| + 2\mu|t|}\\
										&\leqslant C h^{-2N(1+\delta) - 2\mu\delta/\lambda} \leqslant C h^{-3N - \mu/\lambda}. 
\end{align*}
Hence by setting
\[
A_t = \Op_h(b_t)\quad \text{and} \quad R_t = e^{tX} R_0' e^{-tX}  + e^{tX} \Op_h(b) e^{-tX} - A_t,
\]
the proof is complete.
\end{proof}

Taking $|t|= \delta |\log h|/\lambda$, we observe that
\[
\| e^{-tX} R_t \|_{H^{-N}\to H^N} \leqslant C e^{[2N(1 + 1/\delta)\lambda + \mu ] |t|  },
\]
so that for any $\varepsilon>0$,
\begin{equation}\label{eq:size-remainder-good-times}
\| e^{-tX} R_t \|_{H^{-N}\to H^N} \leqslant C_{N,\varepsilon} \exp\Big[ \left(6N \lambda_{\mathrm{max}}(X) + \mu_{\mathrm{max}}(X) + \varepsilon\right) |t|\Big]. 
\end{equation}

\section{Proof of main theorem}
\label{sec:proof-main}

Using a microlocal partition of unity, we see that Theorem \ref{thm:main} will be a consequence of the following result.
\begin{lemma}\label{lemma:pseudo}
Let $A,B\in \Psi^0$ be two pseudo-differential operators of order $0$ such that 
\begin{equation}\label{eq:wf-petit-bloc}
\bigl(\WF(A)\times \WF(B) \bigr) \cap \Gamma(\varphi) = \emptyset. 
\end{equation}
Also let $\chi\in C^\infty_c(\R)$ be equal to $1$ near $0$. Then, for $N>0$ and $\Re s > 6N\lambda_{\mathrm{max}}(X) + \mu_{\mathrm{max}}(X)$, we have the convergence
\[
\int_0^{+\infty}\chi(t/\varepsilon) A e^{t(X-s)} B \dd t \underset{\varepsilon\to 0}{\longrightarrow} A (s-X)^{-1} B
\]
in the space of bounded operators $ H^{-N} \to H^N$.
\end{lemma}

\begin{proof}
We start by an ellipticity argument. For any $\eta>0$, $X$ is elliptic in the directions $(z, \xi)$ for which $|\langle \xi, X(z) \rangle| > \eta |\xi|$. We can thus decompose
\[
B= B_{\mathrm{ell}} + B_{\mathrm{char}}, 
\]
with $B_{\mathrm{char}}$ microsupported in an arbitrarily small conical neighbourhood of $\{\langle \xi, X \rangle=0\}$, and $B_{\mathrm{ell}}$ microsupported in the region of ellipticity of $X$. 

By ellipticity, we can find a pseudo $C(s)$ of order $-2N$, with same wavefront set as $B_{\mathrm{ell}}$, and $R$ a smoothing operator, such that 
\[
B_{\mathrm{ell}} = (s-X)^{2N} C(s) + R.
\]
Then we have 
\begin{align*}
\int_0^{+\infty} \chi(t/\varepsilon) e^{t(X-s)} B_{\mathrm{ell}} \dd t &= \int_0^{+\infty} \chi(t/\varepsilon) e^{t(X-s)} R \dd t + (s-X)^{2N-1}C(s) \\
	& + \int_0^{+\infty} \varepsilon^{-2N} \chi^{(2N)}(t/\varepsilon) e^{t(X-s)} C(s) B_{\mathrm{ell}} \dd t. 
\end{align*}
The first integral in the RHS certainly converges in the space of bounded operators $H^{-N}\to H^N$; the second term is pseudo-differential with wavefront set contained in $\WF(B)$, which does not intersect $\WF(A)$ by assumption, so that when multiplying by $A$ on the left, we obtain a smoothing operator. Finally, for the third term, we can simply use norm bounds (since $C(s)$ is bounded from $H^{-N}$ to $H^N$), to find that its norm $H^{-N}\to H^N$ is controlled (for any $\eta>0$) by 
\[
\int_0^{+\infty} \varepsilon^{-2N} |\chi^{(2N)}(t/\varepsilon)|  e^{-t \Re s} e^{ t (\mu_{\mathrm{max}}+\eta)} \dd t. 
\]
Since the support of $\chi^{(2N)}$ is contained in $\R_+^*$, this tends to $0$ as $\varepsilon\to0$.

Without loss of generality we can now assume that $B=B_{\mathrm{char}}$. By Theorem \ref{thm:egorov}, we can write
\[
A e^{-tX} B = e^{-tX} A_{t}B + e^{-tX} R_{t} B.
\]
The second term in the right hand side is a smoothing operator satisfying
\begin{align*}
\|e^{-tX} R_{t} B\|_{H^{-N} \to H^N} &\leqslant \|e^{-tX}R_t\|_{H^{-N} \to H^N} \|B\|_{H^N \to H^N} \\
\intertext{which is, according to \eqref{eq:size-remainder-good-times}, for $\varepsilon>0$}
&\leqslant C_{N,\varepsilon} e^{(6N\lambda_{\mathrm{max}}(X) + \mu_{\mathrm{max}}(X) + \varepsilon) |t|}.
\end{align*}
This ensures the announced convergence, so that we can turn to the contribution from $A_t$.

The condition \eqref{eq:wf-petit-bloc} ensures that $\WF( A_{t}B ) = \emptyset$. More precisely, there exists $\eta>0$ such that the $\eta$-neighbourhood of $\WF(B)$ does not intersect $\WF(A_{t})$ for each $t \geqslant 0$. This gives, because the family $(A_t)_{t \leqslant \delta |\log h| / \lambda}$ is uniformly bounded in $\Psi^{0}_{\delta, \, \mathrm{sc}}$,
\[
A_tB = \mathcal O(h^\infty) \Psi^{-\infty} = \mathcal{O}_{H^{-N}\to H^N}(e^{-Ct\lambda}) .
\]
This concludes the proof of Lemma \ref{lemma:pseudo}.
\end{proof}

\section{Application to integration currents}
\label{sec:proof-integration-current}

In this section, we will prove Theorem \ref{thm:currents}.
\begin{proof}[Proof of Theorem \ref{thm:currents}]
Point (i) is elementary, so we concentrate on (ii). By virtue of Theorem \ref{thm:main}, Theorem \ref{thm:currents} will be a consequence of the following elementary result, applied with $N = M \times M$, $Q = \Delta$ and 
$\psi_t = \varphi_t^\mathrm L$.
\begin{lemma}\label{lem:elementary}
Let $N$ be a manifold, $(\psi_t)$ a flow on $N$ and $Y$ its generator. Let $Q$  a closed submanifold of $N$ which is transverse to the flow. Let $T \in \R$ and $\varepsilon > 0$ small. Then $\varrho \in C^\infty_c(\left]T-\varepsilon, T+\varepsilon\right[)$. Then the current
\[
\int_{T-\varepsilon}^{T + \varepsilon} \varrho(t)\iota_Y \psi_t^*([Q]) \dd t
\]
coincides with $\pi_{N*}\bigl(\varrho [Q_\varepsilon] \bigr)$ where
\[
Q_\varepsilon = \bigl\{(t,\psi_{-t}(x))\ |\ t \in \left]T-\varepsilon, T+ \varepsilon\right[,~ x \in Q\bigr\} \subset \R \times N
\]
and $\pi_N : \R \times N \to N$ is the projection over the second factor. 
\end{lemma}
\begin{proof}
Up to reparameterizing and choosing an appropriate system of coordinates $(z^1, \dots, z^r)$ we may assume $N = \R^r$, $Y = \partial_{z^1}$, $T = 0$ and $Q = \{z^1 = \cdots = z^k = 0\}.$ Then $[Q]$ is given by
\[
\delta(z^1, \dots, z^k) \dd z^1 \wedge \cdots \wedge \dd z^k.
\]
For each $\omega \in \Omega^\bullet_c(\R^r)$ it holds 
\[
\int_{\R^r} \left(\int_{-\varepsilon}^{ \varepsilon} \varrho(t)\iota_Y \psi_t^*([Q]) \dd t \right) \wedge \omega = \int_\varepsilon^\varepsilon \dd t\varrho(t) \int_{\R^r} \psi_t^*(\iota_Y[Q]) \wedge \omega.
\]
We can write
\[
\psi_t^*(\iota_Y[Q]) = \delta(z^1 + t, z^2, \dots, z^k) \dd z^2 \wedge \cdots \wedge \dd z^k.
\]
Now we take $\omega$ of the form
\[
\omega = f(z^1, \dots, z^n) \dd z^1 \wedge \dd z^{k+1} \wedge \cdots \wedge \dd z^n.
\]
Then we have 
\[
\int_{\R^r} \psi_t^*(\iota_Y[Q]) \wedge \omega = \int_{\{0\} \times \R^{r-k}} f(-t, 0, \dots 0, z^{k+1}, \dots, z^{r}) \dd z^{k+1} \wedge \cdots \wedge \dd z^r.
\]
Hence we get
\[
\int_{\R^r} \left(\int_{-\varepsilon}^{ \varepsilon} \varrho(t)\iota_Y \psi_t^*([Q]) \dd t \right) \wedge \omega  = \int_{Q_\varepsilon} \varrho \, \pi_N^*\omega,
\]
which concludes the proof of the Lemma\dots
\end{proof}
\dots and of Theorem \ref{thm:currents}.
\end{proof}

\section{Geodesic flow on anosov surfaces}
\label{sec:proof-lefschetz}

In this section we prove the main result of our article. Beforehand, we will prove a more general statement valid for every Anosov flow.

 Assume that $(\varphi_t)$ is an Anosov flow on $M$ generated by $X$, that is, for each $z\in M$ there is a $\dd \varphi_t$-invariant splitting
\[
T_zM = \R X(z) \oplus E_u(z) \oplus E_s(z)
\]
which depends continuously on $z$, and such that
\[
\left \|\dd \varphi_{-t}|_{E_u(z)} \right\| + \left \|\dd \varphi_{t}|_{E_s(z)} \right\| \leqslant Ce^{-Ct}, \quad t \geqslant 0,
\]
for some $C > 0$. We also have a decomposition 
\[
T^*M = E_0^* \oplus E_u^* \oplus E_s^*,
\]
where $E_0^*, E_u^*$ and $E_s^*$ are defined by the relations
\[
E_0^*(E_u \oplus E_s) = 0, \quad E_u^*(E_u \oplus \R X) = 0 \quad \text{and} \quad E_s^*(E_s \oplus \R X) = 0. 
\]
Then it follows from the results of \cite{butterley2007smooth} (see also \cite{faure2011upper}) that the resolvent 
\[
\mathrm{R}(s) = (s - \mathcal L_X)^{-1} : \Omega^\bullet(M) \to \mathcal{D}'^\bullet(M)
\]
admits a meromorphic continuation to the whole complex plane; the set of poles is called the set of Ruelle resonances of $(\varphi_t)$, and for each resonance $s_0$, we have a development
\begin{equation}\label{eq:dev}
\mathrm{R}(s) = \mathrm{R}_{\mathrm{hol}}(s) + \sum_{j=1}^{J}\frac{(\mathcal L_X)^{j - 1} \Pi_{s_0}}{(s - s_0)^j}
\end{equation}
where $\mathrm{R}_{\mathrm{hol}}$ is holomorphic near $s = s_0$, and 
\[
\Pi_{s_0} = \frac{1}{2 i \pi} \int_{\mathscr C_{s_0}} \mathrm{R}(s) \dd s
\]
is a finite rank projector, where $\mathscr C_{s_0}$ is a small circle around $s_0$. Moreover, by \cite{dyatlov2016dynamical}, we have the bound
\begin{equation}\label{eq:wf'resolv}
\WF'(\mathrm{R}(s)) \subset \Gamma(\varphi).
\end{equation}

Given any operator $A : \Omega^\bullet(M) \to \mathcal{D}'(M)$ which satisfies the condition $\WF(A) \cap N^*\Delta = \emptyset$, we denote by $\tr^\flat_\mathrm s A$ the super flat trace of $A$, which is given by 
\[
\tr^\flat_\mathrm s A = \int_M \iota_\Delta^* \mathrm{K}_A
\]
where $\iota_\Delta : M \to M \times M$ is given by $\iota_\Delta(x) = (x,x).$ If $A$ is a finite rank operator, then it holds
\[
\tr_\mathrm s A = \tr^\flat_\mathrm s A,
\]
where $A$ is the supertrace of $A$, which is given by 
\[
\tr_\mathrm s A = \sum_{k=0}^n (-1)^k \tr A|_{\Omega^k}.
\]

\begin{theorem}\label{thm:resanosov}
Let $X$ be an Anosov flow, and ${\mathrm{Q}}$ be $n$-dimensional submanifold of $M \times M$ such that $N^\ast {\mathrm{Q}} \cap \Gamma(\varphi)' = \{0\}.$ Let $\widehat {\mathrm{Q}}$ be the operator whose kernel is $[{\mathrm{Q}}]$. Let $s_0$ be a resonance of $X$, and 
\[
\Pi_{s_0} = \frac{1}{2 \pi i} \int_{\mathscr{C}_{s_0}} (s-X)^{-1} \dd s : \Omega^\bullet(M) \to \mathcal{D}'(M)
\]
be the associated spectral projector. Then for any one-form $\omega \in \Omega^1(M)$, we have
\[
\mathrm{Res}_{s=s_0} \int \pi_1^*\omega \wedge \iota_{X^\mathrm{L}}\mathrm{K}(s) \wedge [{\mathrm{Q}}] =  \tr_\mathrm s\left(\omega \iota_X \Pi_{s_0} \widehat {\mathrm{Q}}^\top\right),
\]
where we identified $\omega$ with the operator $u \mapsto \omega \wedge u.$
\end{theorem}
Here, for any operator $A : \Omega^\bullet(M) \to \mathcal{D}'^\bullet(M)$, we denoted by $A^\top$ its adjoint operator, which is defined by
\[
\int_{M} Au \wedge v = \int_M u \wedge A^\top v, \quad u,v \in \Omega^\bullet(M).
\]

\begin{proof}
By \eqref{eq:wf'resolv}, the map 
\[
s \mapsto \int \pi_1^*\omega \wedge \iota_{X^\mathrm{L}}\mathrm{K}(s) \wedge [{\mathrm{Q}}],
\]
which is well defined if $\Re s$ is large, admits a meromorphic extension to the whole complex plane. Certainly, by \eqref{eq:dev}, the residue we are looking for is
\[
\int \pi_1^*\omega \wedge \mathrm{K}_{\iota_{X}\Pi_{s_0}}  \wedge [{\mathrm{Q}}].
\]
However it holds $\pi_1^*\omega \wedge \mathrm{K}_{\iota_{X}\Pi_{s_0}} = \mathrm{K}_{\omega \iota_X \Pi_{s_0}}$, hence the residue coincides with 
\[
\int_{M^2} \mathrm{K}_{\omega \iota_X \Pi_{s_0}} \wedge \mathrm{K}_{\widehat {\mathrm{Q}}}.
\]
Next we state the following
\begin{lemma}\label{lem:simplecomputation}
Let $A, B : \Omega^\bullet(M) \to \mathcal{D}'^\bullet(M)$ two degree preserving operators. Then, it holds
\[
\tr_\mathrm s(AB^\top) = \int_{M^2} \mathrm{K}_A \wedge \mathrm{K}_B,
\]
provided both sides of the equation make sense.
\end{lemma}
This result can be obtained by a computation in local coordinates. This concludes the proof of Theorem \ref{thm:resanosov}.
\end{proof}

Finally, we show that Theorem \ref{thm:lefschetz} can be deduced from Theorem \ref{thm:resanosov} and some results from \cite{Dyatlov-Zworski-17} about the resonant states at $s = 0$ for geodesic flows of surfaces.

\begin{proof}[Proof of Theorem \ref{thm:lefschetz}]
Take the notations of the statement of Theorem~\ref{thm:lefschetz}. The geodesic flow $(\varphi_t)$ is Anosov on $M = S\Sigma$. Then by the results of \cite{Dyatlov-Zworski-17}, we have, near $s=0$,
\begin{equation}\label{eq:laurent}
\mathrm{R}(s) = \frac{\Pi}{s} + Y(s) : \Omega^\bullet(M) \to \mathcal{D}'^\bullet(M),
\end{equation}
where $Y$ is holomorphic near the origin, and 
\[
\Pi = \frac{1}{2 i \pi} \int_{\mathscr C}\mathrm{R}(s) \dd s
\]
where $\mathscr C$ is a small positively oriented circle around $0$; we have 
\[
\mathrm{ran}(\Pi) = \{u \in \mathcal{D}'^\bullet_{E_u^*}(M)~:~\mathcal{L}_X u = 0\}.
\]
Moreover, denoting $C^k = \mathrm{ran}(\Pi) \cap \Omega^k$ and $C^k_0 = C^k \cap \ker \iota_X$, we have a decomposition
\[
C^k = \alpha \wedge C^{k-1}_0 \oplus C^k_0
\]
and in this decomposition, it holds
\begin{equation}\label{eq:pick}
\Pi|_{C^k} = \alpha \wedge \Pi|_{C^{k-1}_0} \oplus \Pi|_{C^k_0}.
\end{equation}
Now if follows from the results of \cite{Dyatlov-Zworski-17} that
\begin{align}
\Pi|_{C^0} &= |M|^{-1}(1 \otimes \alpha \wedge \dd \alpha),\label{eq:pic0} \\
\Pi|_{C^1_0}&= \sum_{j=1}^{2\mathrm g} u_j \otimes \alpha \wedge s_j,\label{eq:pic1} \\
\Pi|_{C^2_0}&= |M|^{-1}(\dd \alpha \otimes \alpha). \label{eq:pic0bis}
\end{align}
Here $|M| = \smallint_M \alpha \wedge \dd \alpha$. Also, for $1\leqslant j \leqslant  2\mathrm g$, $u_j$, $s_j$ are closed $1$-currents whose wavefront set is contained in $E_u^\ast$, $E_s^\ast$ respectively, satisfying 
\begin{equation}\label{eq:ujsj}
\langle u_i, \alpha \wedge s_j \rangle = \delta_{ij}, \quad \iota_X u_j = \iota_X s_j = 0.
\end{equation}
Next, define $\eta(s)$ by \eqref{eq:eta(s)} and
\[
{\mathrm{Q}} = \{(F(z), z)\ |\ z \in M\} = F_\mathrm L(\Delta),
\] 
where $F_\mathrm L(z_1, z_2) = (F(z_1), z_2)$. Then, $\hat{{\mathrm{Q}}} = F^\ast$, and according to Theorem \ref{thm:currents}, $\eta(s)$ coincides with $\iota_{X^L} K(s) \wedge [{\mathrm{Q}}]$. Theorem \ref{thm:resanosov} provides us with the following expression for the residue of $\eta$ at $s=0$:
\[
\int_{M \times M} \iota_{X^\mathrm L}\mathrm{K}_\Pi \wedge [F_\mathrm L(\Delta)] \wedge \pi_1^*\alpha = - \tr_\mathrm s\left(\alpha \iota_X \Pi F_\ast \right).
\]
(we used \eqref{eq:laurent} and Lemma \ref{lem:simplecomputation}). Now we note that $\iota_X \Omega^0 = 0$, which gives us that $- \tr_\mathrm s\left(\alpha \iota_X \Pi F_\ast \right)$ is equal to
\[
\tr\left(F_\ast \alpha \iota_X \Pi|_{\Omega^1}\right) + \tr\left(F_\ast \alpha \iota_X \Pi|_{\Omega^3}\right)- \tr\left(F_\ast \alpha \iota_X \Pi|_{\Omega^2}\right),
\]
where we used the cyclicity of the trace. Now, using \eqref{eq:pick}, \eqref{eq:pic0} and \eqref{eq:pic0bis}, one can see that
\[
\tr\left(F_\ast \alpha \iota_X \Pi|_{\Omega^1}\right) = |M|^{-1} \int_M F_\ast \alpha \wedge \dd \alpha
\]
and $\tr\left(F_\ast \alpha \iota_X \Pi|_{\Omega^3}\right) = 1$.
Finally, Equation \eqref{eq:pic1} yields
\[
\tr\left(F_\ast \alpha \iota_X \Pi|_{\Omega^2}\right) = \sum_j \int u_j \wedge \alpha \wedge F^*s_j.
\]
Recall that the pull-back map $\pi^* : \Omega^1(\Sigma) \to \Omega^1(S\Sigma)$ induces an isomorphism $H^1(\Sigma) \to H^1(S\Sigma)$. Thus, by ellipticity of the exterior differential $\dd : C^\infty(M) \to \Omega^1(M),$ one can find smooth closed forms $\omega_1, \dots,  \omega_{2\mathrm g} \in \Omega^1(\Sigma)$ generating $H^1(\Sigma)$ such that
\begin{equation}\label{eq:pi*}
s_j = \pi^*\omega_j + \dd g_j
\end{equation}
for some $g_j \in \mathcal{D}'_{E_u^*}(M).$  Using \eqref{eq:pi*} one gets
\[
F^*s_j = F^* \pi^*\omega_j + \dd F^*g_j = \pi^*f^*\omega_j + \dd F^*g_j.
\]
Now we have
\[
f^*\omega_j = \sum_i A_{ij} \omega_i + \dd \theta_j
\]
for some $\theta_j \in C^\infty(\Sigma)$, where we denoted by $A = (A_{ij})$ is the matrix of $f^* : H^1(\Sigma) \to H^1(\Sigma)$ in the basis $(\omega_j)$. Thus we get
\[
F^*s_j = \sum_i A_{ij} s_i + \dd \left(F^*g_j + \pi^*\theta_j - \sum_i A_{ij}g_i\right).
\]
Since $u_j$ is closed, we conclude with \eqref{eq:ujsj} that
\[
\sum_j \int u_j \wedge \alpha \wedge F^*s_j = \tr f^*|_{H^1(\Sigma)},
\]
which concludes the proof of Theorem \ref{thm:lefschetz}.
\end{proof}

\section{Examples of maps transverse to an Anosov flow}
\label{sec:examples}

Let us start by translating the assumption into more easily understable geometric statements. We will assume that
\[
{\mathrm{Q}} = \{ (F(z), z) \ |\ z\in M\}
\]
is the graph of a smooth map $F$. When $X$ is an Anosov vector field, 
\[
\overline{\Omega_+} = \Omega_+ \cup E_u^\ast \times E_s^\ast. 
\]
In particular, one gets the following characterization.
\begin{lemma}
$N^\ast {\mathrm{Q}} \cap \Gamma(\varphi)' = \{0\}$ if and only if three conditions are satisfied
\begin{enumerate}[label=\emph{(\roman*)}]
	\item \label{item:1} $F$ has non-degenerate fixed points
	\item \label{item:2} $\R^+\times {\mathrm{Q}}$ is transverse to $G_\varphi$. 
	\item \label{item:3} $\dim\bigl(\dd F(E_s\oplus E_0) \cap (E_u\oplus E_0) \bigr)\leqslant 1$ everywhere. 
\end{enumerate} 
\end{lemma}
Observe that both \emph{\ref{item:1}} and \emph{\ref{item:2}} are generic. However \emph{\ref{item:3}} is not so easily verified. Indeed, it means that the graph of $F$ is transverse to every leaf of the foliation
\[
M\times M = \bigcup_{x,y} \left(W^{u0}(x)\times W^{s0}(y)\right). 
\]
Generically in 3 dimensions, the points where the graph of $F$ is not transverse is a discrete set of curves. In particular, to find an example, we have to ensure by hand that \emph{\ref{item:3}} is satisfied. Thankfully, \emph{\ref{item:3}} is an open condition. 
\begin{proof}
Condition \emph{\ref{item:1}} stems from the fact that $(\Delta(T^\ast M))' = N^\ast(\Delta(M))$, so that $N^\ast {\mathrm{Q}} \cap (\Delta(T^\ast M))'  = 0$ if and only if ${\mathrm{Q}}$ is transverse to the diagonal, which is equivalent to saying that the fixed points of $F$ are non-degenerate.

For condition \emph{\ref{item:2}}, it suffices to observe that $N^\ast G_\varphi$ is equal to
\[
\{ (t,\varphi_t(z),z; -\lambda(X), \dd \varphi_t(z)^{-\top}\xi +\lambda , -\xi) \ |\ \lambda\in E_0^*,\ \xi(X)=0\},
\]
and
\[
N^\ast(\R^+\times {\mathrm{Q}})=\{(t,0) \ |\ t\geqslant 0\} \times N^\ast {\mathrm{Q}}. 
\]
From this it follows that
\[
\pi\left[ N^\ast(\R^+\times {\mathrm{Q}}) \cap N^\ast G_\varphi \right]  = N^\ast {\mathrm{Q}} \cap \Omega_+', 
\]
(here $\pi : T^\ast (\R\times M\times M) \to T^\ast (M\times M)$ is the natural projection). 

Finally, for \emph{\ref{item:3}}, we study the condition
\[
N^\ast {\mathrm{Q}} \cap (E_u^\ast \times E_s^\ast) = 0. 
\]
Since $E_u^\ast = (E_0 \oplus E_u)^\perp$, $E_s^\ast = (E_0\oplus E_s)^\perp$, this means that for every $(x,y)\in {\mathrm{Q}}$, 
\[
N^\ast {\mathrm{Q}} \cap \left[ N^\ast \left(W^{u0}_{\mathrm{loc}}(x) \times W^{s0}_{\mathrm{loc}}(y)\right) \right] = 0.
\]
That is, ${\mathrm{Q}}$ is tranverse to the foliation $W^{u0}\times W^{s0}$. Equivalently, 
\[
\{ (\dd F(u), u) \ |\ u\in T_z M\}  \pitchfork \{ (v, w) \ |\ v\in E_u\oplus \R X,\ w\in E_s \oplus \R X\}.
\] 
Denoting $v_{0,s,u}$ the components of $v$ along $E_0,E_s,E_u$, this is equivalent to requiring that
\[
\{ (\dd F(z) u_s, u_u) \ |\ u\in T_z M \} = E_s(F(z))\times E_u(z)
\]
This is equivalent to $\dd F(z)_s : E_s \oplus E_0 \to E_s$ being onto, which reads
\[
\dim(\dd F(E_s \oplus E_0) \cap (E_u\oplus E_0)) \leqslant 1.
\]
\end{proof}

In dimension 3, condition \emph{\ref{item:3}} is somewhat simpler: it means that we never have $\dd F(E_s\oplus E_0) \subset E_u\oplus E_0$. 
\begin{corollary}
Let $f:\Sigma\to \Sigma$ be a diffeomorphism. If $f$ sufficiently close to identity, we can find a bundle map $F$ over $f$, that is transverse to the flow. 
\end{corollary}

\begin{proof}
Since $f$ need not have non-degenerate fixed points, we have to assume that $F$ has no fixed points. However we set
\[
F_\theta(x,v) = R_{\theta}\left(f(z), \frac{\dd f(z)v}{|\dd f(z)v|}\right). 
\]
For any $\theta \notin 2\pi \Z$, if $f$ sufficiently close to identity, $F$ has no fixed points. Now, if $\theta$ is sufficiently small, condition \emph{\ref{item:3}} is also satisfied, because $F$ is close to identity. 
\end{proof}

\begin{proposition}\label{prop:transverse}
For every simple free homotopy class $h$ of $\Sigma$, we can find a metric $g$, a diffeomorphism $f$ which is a Dehn twist around the closed geodesic $\gamma\in h$, and a bundle map $F$ over $f$, transverse to the flow. 
\end{proposition}

If $\gamma_1,\dots,\gamma_{3\mathrm{g}-3}$ are pairwise disjoint simple closed geodesics ($\mathrm{g}$ the genus of $\Sigma$), we can make the same construction near each such geodesic (i.e pinch them and make a Dehn twist, or a power of a Dehn twist). In this way we generate a maximal rank free abelian subgroup of the mapping class group of $\Sigma$.

\begin{proof}
Since the genus of $\Sigma$ is $\mathrm{g}\geqslant 2$, we can find a constant curvature $-1$ metric $g_\ell$ on $\Sigma$, such that the closed geodesic $\gamma\in h$ has length $\ell>0$. By the Collar Lemma \cite[Theorem 4.1.1]{Buser}, provided $\ell>0$ is small enough, $\gamma$ has a neighbourhood of the form $(\R/\ell \Z)_\tau \times [-2, 2]_\rho$ where the metric writes
\[
\dd s^2 = h(\rho)^2\dd \tau^2 + \dd \rho^2,
\]
where 
\[
h(\rho) = \cosh \rho.
\]
Let $H = (R_{-\pi / 2})_* X$ and $V$ be the generator of the flow $(R_\theta)_{\theta \in \R}$.
In the coordinates $(\partial_\tau, \partial_\rho, \partial_\theta)$, we have $V = \partial_\theta$,
\[
X = 
\begin{pmatrix} h^{-1}\cos \theta \\ \sin \theta \\ 4 \cos \theta \, \tanh \rho \end{pmatrix} \quad \text{ and } \quad
H = 
\begin{pmatrix} -h^{-1}\sin \theta \\ \cos \theta \\ -4 \sin \theta \, \tanh \rho \end{pmatrix}.
\]
Moreover, because $\kappa = -1$, it is a classical result that
\begin{equation}\label{eq:stablehyperbolic}
E_s = \R(H - V) \quad \text{and} \quad E_u = \R(H + V).
\end{equation}
On the other hand, let $\chi \in C^\infty([-2, 2], [0, 1])$ be a monotone function such that $\chi \equiv 0$ on $[-2, -1]$ and $\chi \equiv 1$ on $[1, 2]$, and put
\[
f_0(\tau, \rho) = (\tau + \ell \chi(\rho), \rho).
\]
Now set 
\[
F_0(\tau, \rho, v) = \left(f_0(\tau, \rho), \frac{\dd f_0(\tau, \rho) v}{\|\dd f_0(\tau, \rho) v\|}\right).
\]
Then it is not hard to see that, in the coordinates $(\tau, \rho, \theta)$, we have
\[
F_0(\tau, \rho, \theta) = \left(\tau + \ell \chi(\rho), \rho, \varphi \right)
\]
where $\varphi \in \R/2\pi \Z$ is defined by 
\begin{equation}\label{eq:phitheta}
\tan \varphi = \frac{\sin \theta}{\cos \theta + \ell \chi'(\rho)h(\rho)\sin \theta}
\end{equation}
if $\chi'(\rho) \neq 0$ and $\varphi = \theta$ otherwise.
We may thus write
\begin{equation}\label{eq:df}
\dd F_0(\tau, \rho, \theta) = 
\begin{pmatrix}
1 & \ell \chi'(\rho) & 0 \\ 0 & 1 & 0 \\ 0 & \partial_\rho \varphi & \partial_\theta \varphi
\end{pmatrix}.
\end{equation}
For $z = (\tau, \rho, \theta)$ let $\beta_z = (\partial_\tau, \partial_\rho, \partial_\theta)$, and $\tilde \beta_z = (X(z), H(z), K(z))$. Then the matrix of $\tilde \beta_z$ in $\beta_z$ is given by
\[
[\beta_z : \tilde \beta_z] =
\begin{pmatrix}
h^{-1} \cos \theta & -h^{-1} \sin \theta & 0 \\ \sin \theta & \cos \theta & 0 \\ 2 \cos \theta \, \tanh \rho & -2 \sin \theta \, \tanh \rho  & 1
\end{pmatrix},
\]
and its inverse $[\tilde \beta_{F_0(z)} : \beta_{F_0(z)}] $ at $F_0(z) = \left(\tau + \ell \chi(\rho), \rho, \varphi \right)$ reads
$$
\begin{pmatrix}
h \cos \varphi & \sin \varphi & 0 \\ -h \sin \varphi & \cos \varphi & 0 \\ -2 \partial_\rho h & 0 & 1
\end{pmatrix}.
$$
Using \eqref{eq:df} we can compute the matrix of $\dd F_0(z)$ computed in the basis $\tilde \beta_z$ and $\tilde \beta_{F_0(z)}$; it is of the form
$$
\begin{pmatrix}
\lambda & \mu^\perp & 0 \\
0 & \mu & 0 \\
\delta & \varepsilon & \nu
\end{pmatrix}
$$
for some functions $\lambda, \mu^\perp, \mu, \delta, \varepsilon, \nu$ of $(\tau, \rho, \theta)$, with $\nu = \partial_\theta \varphi$ and
$$
\begin{aligned}
\mu &= \sin \theta \sin \varphi + \cos \theta(-\ell h \chi' \sin \varphi + \cos \varphi), \\
\varepsilon &= -2 \partial_\rho h(-h^{-1} \sin \theta + \ell \chi' \cos \theta) - 2 h^{-1} \partial_\rho h \sin \theta \partial_\theta \varphi + \partial_\rho \varphi \cos \theta.
\end{aligned}
$$
Now, write
$$
\varepsilon = 2 h^{-1} \partial_h \sin \theta (1 - \partial_\theta \varphi) + \partial_\rho \varphi \cos \theta - 2 \ell \chi' \partial_\rho h \cos \theta.
$$
Equation \eqref{eq:phitheta} yields
\[
\partial_\theta \varphi \to 1, \quad \partial_\rho \varphi \to 0 \quad \text{and} \quad \varphi \to \theta
\]
uniformly on $[-2, 2] \times \R/\ell \Z$ as $\ell \to 0$. Therefore we get
\[
\varepsilon \to 0 \quad \text{and} \quad \mu, \nu \to 1
\]
as $\ell \to 0.$ Consequently, whenever $\ell$ is small enough, we have 
\[
\dd F_0(E_u \oplus E_0) \neq E_s \oplus E_0
\]
by \eqref{eq:stablehyperbolic}, which is equivalent to the condition \emph{\ref{item:3}}. Since \emph{\ref{item:3}} is an open condition, any sufficiently small perturbation of $F_0$ will still satisfy it. 

We can find $f$ arbitrarily close to $f_0$, with non degenerate fixed points $x_1, \dots, x_q$. Let $F_1$ be a bundle map over $f$, close to $F_0$, so that $F_1$ satisfies \emph{\ref{item:3}}. 

Next, consider the function
\[
\mathscr F : \R/2\pi\Z \times M \to M \times M
\]
defined by
\[
\mathscr F(\theta, z) = (R_\theta  F_1(z), z),
\]
and set $F_\theta(z) = R_\theta \circ F_1$. We claim that for a dense set of $\theta$'s, $F_\theta$ satisfies \emph{\ref{item:1}} and \emph{\ref{item:2}}. For this we will use the tranversality result \cite[Th\'eor\`eme p. 93]{laudenbach2012transversalite}.

For point \emph{\ref{item:1}}, observe that $F_\theta$ has non degenerate fixed points if and only if $F_\theta$ has non-degenerate fixed points in each fiber $S_{x_j} \Sigma$, $j=1\dots q$. However, 
\[
\mathrm{ran} \left(\dd\mathscr{F}_{|\R\times S_{x_j}\Sigma}\right) = T [S_{x_j}\Sigma] ^2
\]
is tranverse to the diagonal, so that according to \cite[Th\'eor\`eme p. 93]{laudenbach2012transversalite}, for almost every $\theta$, $F_\theta$ has non degenerate fixed points in $S_{x_j} \Sigma$. 

Let us turn to point \emph{\ref{item:2}}. Define for $\tau >0$ 
\[
G_\tau = \{(\varphi_t(z), z)\ |\ 0 \leqslant t < \tau,~z \in M\} \subset M \times M.
\]
Let us prove that $\mathscr{F}$ is transverse to $G_\tau$. Let $z \in M$, $\theta \in \R/2\pi \Z$ and $\tau > t\geqslant 0$ such that we have $R_\theta F_1(x,v) = \varphi_t(z)$. For simplicity we set $z_t = \varphi_t(z)$. If $t=0$, the tranversality of $\mathscr{F}$ with $G_\varphi$ follows from the tranversality with the diagonal, so we assume that $t>0$. At our point of intersection, $\mathrm{ran} (\dd\mathscr{F})$ contains 
\[
T_{z_t} S_{f(x)} \Sigma \times T_{z} S_x \Sigma = \mathrm{vect} \{ (V,0),\ (0,V) \} 
\]
On the other and, $T G_\tau$ is generated by $(X, 0)$ and the vectors
\[
(\dd \varphi_t(z)W, W), \quad W \in T_{z} M.
\]
In particular, $T G_\tau$ contains also $(0,X)$. Finally, since $\Sigma$ has no conjugate points, $\dd\varphi_t(V)$ always has a non-zero component along $H$ when $t>0$, so that $\mathrm{ran} (\dd\mathscr{F}) + T G_\tau$ contains 
\[
(X,0),\ (0,X),\ (V,0),\ (0,V),\ (H,0),\ (0,H),
\]
which generate $T_{(z_t, z)}(M \times M)$. Taking a sequence $\tau=n$, $n\in \mathbb{N}$, we can use again \cite[Th\'eor\`eme p. 93]{laudenbach2012transversalite} to conclude. This closes the proof of Proposition \ref{prop:transverse}.
\end{proof}

\bibliographystyle{alpha}
\bibliography{biblio}

\end{document}